\documentclass[twoside,11pt,reqno]{amsart}
\usepackage{amsmath,amssymb,amscd,mathrsfs, todonotes,appendix}
\usepackage{graphics,verbatim}
\usepackage{todonotes}
\usepackage{enumitem}
\usepackage{hyperref}
\usepackage{setspace} 

\usepackage{xcolor}

\usepackage{float} 
\usepackage[labelformat=empty]{caption}
\usepackage{fancybox} 

\newcommand{\losemi}{{\otimes \kern -.78em \ltimes}}
\newcommand{\rosemi}{{\otimes \kern -.78em \rtimes}}

\renewcommand{\a}{\alpha}

\newcommand{\Hh}{\operatorname{H}}
\newcommand{\HH}{\operatorname{HH}}

\newcommand{\DGr}{\operatorname{Dist}(G_{r})_{\operatorname{ad}}}
\newcommand{\DGone}{\operatorname{Dist}(G_{1})_{\operatorname{ad}}}
\newcommand{\DBone}{\operatorname{Dist}(B_{1})_{\operatorname{ad}}}

\newcommand{\la}{\lambda}

\newcommand{\opH}{\operatorname{H}}

\makeatletter
\newcommand{\leqnomode}{\tagsleft@true}
\newcommand{\reqnomode}{\tagsleft@false}
\makeatother

\newtheorem{theorem}{Theorem}[subsection]

\makeatletter\let\c@fact\c@theorem\makeatother

\makeatletter\let\c@note\c@theorem\makeatother

\makeatletter\let\c@lemma\c@theorem\makeatother

\makeatletter\let\c@lemma\c@theorem\makeatother

\makeatletter\let\c@quest\c@theorem\makeatother

\newtheorem{prop}{Proposition}[subsection]
\makeatletter\let\c@prop\c@theorem\makeatother

\makeatletter\let\c@conj\c@theorem\makeatother

\makeatletter\let\c@cor\c@theorem\makeatother

\makeatletter\let\c@defn\c@theorem\makeatother

\theoremstyle{definition}

\makeatletter\let\c@remark\c@theorem\makeatother

\makeatletter\let\c@example\c@theorem\makeatother
\numberwithin{equation}{subsection}

\usepackage[capitalise]{cleveref}
\crefname{theorem}{Theorem}{Theorems}
\crefname{fact}{Fact}{Facts}
\crefname{note}{Note}{Notes}
\crefname{lemma}{Lemma}{Lemmas}
\crefname{alg}{Algorithm}{Algorithms}
\crefname{remark}{Remark}{Remarks}
\crefname{example}{Example}{Examples}
\crefname{prop}{Proposition}{Propositions}
\crefname{conj}{Conjecture}{Conjectures}
\crefname{cor}{Corollary}{Corollaries}
\crefname{defn}{Definition}{Definitions}
\crefname{equation}{\!\!}{\!\!} 

\newcounter{listequation}

\begin{document}

\title{On the Hochschild Cohomology for Frobenius Kernels}

\begin{abstract} In this paper the authors investigate the structure of the Hochschild cohomology for Frobenius kernels. The authors first establish some fundamental constructions to compute Hochschild cohomology by using spectral sequences. This enables us to provide a complete description of the $G$-algebra structure of the Hochschild cohomology for the first Frobenius kernel $G_{1}$ where $G=SL_{2}$. This computation heavily relies on the calculation of the adjoint action on the restricted enveloping algebra. 
\end{abstract}

\author{\sc Tek\.{I}n Karada\u{g}}
\address
{Department of Mathematics\\ University of Georgia \\
Athens\\ GA~30602, USA}
\email{karadag@math.uga.edu}

\author{\sc Daniel K. Nakano}
\address
{Department of Mathematics\\ University of Georgia \\
Athens\\ GA~30602, USA}
\thanks{Research of the second author was supported in part by
NSF grants DMS-2101941 and DMS-2401184}
\email{nakano@math.uga.edu}

\maketitle
\section{Introduction} 

\subsection{} Let $G$ be a affine algebraic group scheme defined over ${\mathbb F}_{p}$ and $F:G\rightarrow G$ be the Frobenius morphism. The scheme theoretic 
kernels, $G_{r}$, are obtained by taking the kernel of the $r$th iteration of $F$. The infinitesimal group schemes $G_{r}$ are of major interest in the study of the representation theory of 
$G$, and the representation theory of $G_{r}$ is equivalent to studying the modules for the finite-dimensional Hopf algebra $\text{Dist}(G_{r})$ (cf. \cite[I. Ch. 9]{rags}. 

A natural question is to determine the Hochschild cohomology of $\text{Dist}(G_{r})$ which will be denoted by $\HH^{\bullet}(G_{r})$. Through a series of isomorphisms one has 
$$\HH^{\bullet}(G_{r})\cong\Hh^{\bullet}(G_{r},\DGr) .$$ 
where $\text{Dist}(G_{r})_{ad}$ is the $G_{r}$-module obtained by using the adjoint action of $G_{r}$. 

Friedlander and Suslin \cite{FS97} proved that the cohomology ring $R=\Hh^{\bullet}(G_{r},k)$ is a finitely generated $k$-algebra. One also has that the Hochschild cohomology 
$\HH^{\bullet}(G_{r})$ is a finitely generated module over $R$. Furthermore, $\HH^{\bullet}(G_{r})$ is also a finitely generated $k$-algebra. A natural question to ask is what is the 
structure of $\HH^{\bullet}(G_{r})$ as a $G$-algebra? 

\subsection{} The computation of $R=\Hh^{\bullet}(G_{r},k)$ is still an open for problem for $r\geq 2$ and for $r=1$ when $p<h-1$ (cf. \cite{BNPP14}). The reader is referred to the survey article by Bendel \cite{B24} for 
details pertaining to the history and related problems surrounding this computation. For $r=1$ and $p>h$, the odd degree cohomology of $R$ is zero and the even degree cohomology identifies with the 
ring of regular functions on the nilpotent cone of ${\mathfrak g}=\text{Lie }G$ \cite{AJ84, FP86}. A natural starting point for computing the Hochschild cohomology is to first understand this case. 

The main ingredient in our work will be the use of spectral sequences to compute Hochschild cohomology. Historically, this has not been the standard approach and many authors have involved the use of 
resolutions to make explicit calculations \cite{N13,W19}. In our work, we will later see one of the main obstacles will be to understand the $G$-structure of $\text{Dist}(G_{1})$ which identifies with a truncated polynomial algebra $\overline{S}^{\bullet}({\mathfrak g})$. For quantum groups, the adjoint action on the small quantum groups is much more complicated, and this is the main reason why we will only focus on Frobenius kernels in this paper. 

\subsection{} The paper is organized as follows. In Section~\ref{S:prelim}, the notation and conventions for the paper are introduced. The relevant spectral sequences are described and will be used 
for the computations throughout the paper. We also include a subsection on how these spectral sequences will be used to make the key calculations in the paper. 

In the following section, Section~\ref{S:Taft}, we provide a computation of the $B$-algebra structure of the Hochschild cohomology for $\HH(B_{1})$ for the rank one case, where $B$ is a Borel subgroup of 
$G$. This computation is important to see because it demonstrates the power in using spectral sequences and Kostant's theorem for Lie algebra cohomology to obtain the results efficiently. Prior computations by others have involved explicit resolutions (cf. \cite{N13}). 

The main section of the paper, Section~\ref{S:SL2}, is devoted to providing a complete answer to the $G$-algebra structure of $\HH(G_{1})$ when $G=SL_2$. The computation is quite involved 
due to the fact that one needs knowledge about the action action on $\text{Dist}(G_{1})$. We provide a complete answer to this question for all primes by using results about tilting modules and good filtrations. 
In the appendix (Section~\ref{S:appendix}), we provide tables for $p=2,3,5,7$ that indicate the intricacies of the $G$-module structure on  $\text{Dist}(G_{1})$.

\section{Preliminaries}\label{S:prelim}

\subsection{Notation}\label{S:notation} In this paper we will generally follow the standard conventions in \cite{rags}. Throughout this paper, $k$ is an algebraically closed field of characteristic $p>0$. Let
\vskip .25cm 
\begin{itemize}
\item[$\bullet$] $G$ be a connected semisimple algebraic group scheme defined over ${\mathbb F}_{p}$.
\item[$\bullet$] $T$ be a fixed split maximal torus in $G$. 
\item[$\bullet$] $\Phi$ be the root system associated to $(G,T)$. 
\item[$\bullet$]  $\Phi^{\pm}$ be the set of positive (resp. negative) roots. 
\item[$\bullet$]  $\Delta=\{\alpha_1,\dots,\alpha_{l}\}$ be the set of simple roots determined by $\Phi^+$. 
\item[$\bullet$] $B$ be the Borel subgroup given by the set of negative roots, $U$ be the unipotent radical of $B$. 
\item[$\bullet$] $W$ be the Weyl group associated with $\Phi$. 
\item[$\bullet$]  $h$ denote the Coxeter number for the root system associated to $G$, i.e., $h = \langle\rho,\alpha_0^{\vee}\rangle + 1$.
\item[$\bullet$] $X:=X(T)$ be the integral weight lattice spanned by the fundamental weights $\{\omega_1,\dots,\omega_l\}$. 
\item[$\bullet$] $X^{+}$ denote the dominant weights for $G$.  \item[$\bullet$] $X_{r}$ be the $p^{r}$-restricted weights. 
\item[$\bullet$] $\leq$ be the order relation defined on $X$ via $\mu \leq \la$ iff $\la - \mu = \sum_{\a \in \Delta}n_{\a}\a$ for $n_{\a} \in {\mathbb Z}_{\geq 0}$.
\end{itemize} 

For $\lambda\in X^{+}$, there are four fundamental families of finite-dimensional rational $G$-modules: 
\begin{itemize} 
\item[$\bullet$] $L(\lambda)$ (simple), 
\item[$\bullet$] $\nabla(\lambda)$ (costandard/induced), 
\item[$\bullet$] $\Delta(\lambda)$ (standard/Weyl), 
\item[$\bullet$] $T(\lambda)$ (indecomposable tilting).
\end{itemize} 

Let $F^{r}:G\rightarrow G$ be the $r$th iteration of the  Frobenius morphism, and $G_{r}$ be the scheme theoretic kernel of this map which is often called the 
$r$th Frobenius kernel. One can restrict $F$ to the subgroups $B$, $T$ and $U$ to obtain $B_{r}$, $T_{r}$ and $U_{r}$. Set $G_{r}T=(F^{r})^{-1}(T)$. 

Let $\text{Mod}(G)$ be the category of rational $G$-modules. Given $M\in \text{Mod}(G)$, one can compose the representation map with $F^{r}$ to obtain a 
module $M^{(r)}$ where $G_{r}$ acts trivially and the $G$ action is obtained via inflation. Moreover, if $N$ is a $G/G_{r}\cong G^{(r)}$-module then one can untwist the action 
$N^{(-1)}$ to obtain a $G$-module. 

The category of $G_{r}$-modules is equivalent to the category of modules for $\text{Dist}(G_{r})$ which is a finite-dimensional cocommutative Hopf algebra. 
Let $Q_{r}(\lambda)$ denote the projective cover (equivalently, injective hull) of $L(\lambda)$ as a $G_{r}$-module, $\lambda\in X_{r}$.  

\subsection{Spectral Sequences}\label{S:spectral}  In this section, we will present the spectral sequences that will be used for our computational purposes for $\HH(G_{r})$. Note that 
the spectral sequences hold in higher generality for different algebraic group schemes and more general modules. For this paper, we have chosen to state the theorem for specific modules 
that are germane to our situation. Throughout this section, $G$ will be a reductive algebraic group scheme and $B$ will be a Borel subgroup scheme in $G$. 

The first two spectral sequences hold for all $r\geq 1$ and all primes $p>0$. 

\begin{theorem}\label{T:spectral1}  Let $G$ be a reductive algebraic group scheme and $B$ be a Borel subgroup scheme. There exists a first quadrant spectral sequence 
$$E_{2}^{i,j}=R^{i}\operatorname{ind}_{B}^{G} \Hh^{\bullet}(B_{r},\DGr) \Rightarrow \HH^{i+j}(G_{r})$$ 
\end{theorem} 

\begin{theorem}\label{T:spectral2}  Let $B$ be a Borel subgroup scheme with $B=T\ltimes U$ where $T$ is a maximal torus and $U$ is the unipotent radical of $B$. 
\begin{itemize} 
\item[(a)] There exists a spectral sequence 
$$E_{2}^{i,j}=\Hh^{i}(T_{r},\Hh^{j}(U_{r},\DGr))\Rightarrow \Hh^{i+j}(B_{r},\DGr).$$
\item[(b)] One has an isomorphism of $B$=modules, 
$$ \Hh^{\bullet}(B_{r},\DGr)\cong \Hh^{\bullet}(U_{r},\DGr)^{T_{r}}.$$ 
\end{itemize} 
\end{theorem} 

There is a spectral sequence that can be used to compute $\Hh^{\bullet}(U_{r},\DGr)$ when $p\geq 3$. However, when $r>2$, the spectral sequence cannot be regraded to make it in the first quadrant. This causes difficulties in making explicit computations. We state the case when $r=1$ where the spectral sequence is in the first quadrant, allowing us to make 
explicit computations. Also included in the discussion is an outline of the strategy for the use of the spectral sequence machinery.

\begin{theorem}\label{T:spectral3} Let $G$ be a reductive algebraic group scheme, with $r=1$ and $p\geq 3$. 
\begin{itemize} 
\item[(a)] There exists a first quadrant spectral sequence 
$$E_{2}^{2i,j}=S^{i}({\mathfrak u}^{*})^{(1)}\otimes \Hh^{j}({\mathfrak u},\DGone)\Rightarrow \Hh^{i+j}(U_{1},\DGone).$$ 
\item[(b)] There exists a first quadrant spectral sequence 
$$E_{2}^{2i,j}=S^{i}({\mathfrak u}^{*})^{(1)}\otimes \Hh^{j}({\mathfrak u},\DGone)^{T_{1}}\Rightarrow \Hh^{i+j}(B_{1},\DGone).$$ 
\end{itemize} 
\end{theorem} 

\subsection{Strategy for the Computations} The spectral sequences in Section~\ref{S:spectral} will be used as the primary apparatus to make computations for 
$\HH^{\bullet}(G_{r})$. When one replaces $\DGone$ by $k$, the cohomology ring $\Hh^{\bullet}(G_{1},k)$ for $p>h$ was determined using the strategy below  
by Andersen-Jantzen \cite{AJ84}. 

Ideally, one would like to prove that 
\begin{equation} \label{eq:induction} 
\HH^{\bullet}(G_{r})\cong \text{ind}_{B}^{G} \Hh^{\bullet}(B_{r},\DGr). 
\end{equation} 
Note that this is related to the induction conjecture stated in \cite[Section 2.3]{B24}. 

Set $r=1$ and $p\geq 3$. In order to prove (\ref{eq:induction}), one starts with analyzing the spectral sequence in Theorem~\ref{T:spectral3}(b). The key idea is to 
understand the structure of $\Hh^{\bullet}({\mathfrak u},\DGone)^{T_{1}}$. With enough information, one hopes to show that the spectral sequence in 
Theorem~\ref{T:spectral3}(b) collapses for $p>h$. This would yield the isomorphism 
\begin{equation}\label{E:B1coho} 
\Hh^{\bullet}(B_{1},\DGone)\cong  S^{\bullet}({\mathfrak u}^{*})^{(1)}\otimes \Hh^{\bullet}({\mathfrak u},\DGone)^{T_{1}}. 
\end{equation} 
This is true when  $\DGone$ is replaced by $k$. However, we will show that (\ref{E:B1coho}) does not hold for $G=SL_{2}$.  In fact, the spectral sequence 
Theorem~\ref{T:spectral3}(b) collapses at $E_{3}$, and one can show that 
\begin{equation}\label{E:B1coho} 
\Hh^{\bullet}(B_{1},\DGone)\cong  [S^{\bullet}({\mathfrak u}^{*})^{(1)}\otimes \Hh^{\bullet}({\mathfrak u},\DGone)^{T_{1}}]/I_{p}. 
\end{equation} 
See Theorem~\ref{T:B1-coho} for a description of $I_{p}$.  Note that at this point of stopping one could only conclude that this isomorphism holds as $S^{\bullet}({\mathfrak u}^{*})^{(1)}$-modules. In order to deduce that this is 
a ring isomorphism one needs to degrade the spectral sequence and use the techniques given by Drupieski, Nakano and Ngo (cf. \cite{DNN12}). 

The final step involves using a Grauert-Riemenschnieder type vanishing result, namely $R^{j}\text{ind}_{B}^{G}\ S^{\bullet}({\mathfrak u}^{*})\otimes \lambda=0$ for $\lambda \in X_{+}$ and $j>0$. 
If one can show that the weights of $\Hh^{\bullet}(B_{1},\DGone)$ have factors of the form $S^{\bullet}({\mathfrak u}^{*})\otimes \lambda$ ($\lambda\in X_{+}$), along some additional dominant weights
then it would follow that (\ref{eq:induction}) holds and as rings for $p>h$, 
\begin{equation} 
\HH^{\bullet}(G_{1})\cong \text{ind}_{B}^{G}\ [S^{\bullet}({\mathfrak u}^{*})^{(1)} \otimes \Hh^{\bullet}({\mathfrak u},\DGone)^{T_{1}}]/I_{p}. 
\end{equation}

\section{Taft algebra computations}\label{S:Taft}

In this section, we will compute the Hochschild cohomology for $B_{1}$ for $G=SL_{2}$. Our approach will utilize spectral sequences and we note that 
prior calculations of this type for the Taft algebra used techniques involving resolutions. [cite]
 
\subsection{} First note that the Hochschild cohomology $\HH^{\bullet}(B_{r})$ identifies with $\Hh^{\bullet}(B_{1},\DBone)$. The same 
arguments given in Section~\ref{S:spectral} can be used to construct the following spectral sequence.   

\begin{theorem} Let $p\geq 3$ and let $B$ be a Borel subgroup of $G$ corresponding to the negative roots. Then there exists a first quadrant spectral sequence 
$$E_{2}^{2i,j}=S^{i}({\mathfrak u}^{*})^{(1)}\otimes \Hh^{j}({\mathfrak u},\DBone)^{T_{1}}\Rightarrow \HH^{i+j}(B_{1}).$$ 
\end{theorem} 

\subsection{Rank One Case, $p\geq 3$} When $G=SL_{2}$, one has an isomorphism of $B_{1}$-modules under the adjoint representation: 
$$\DBone\cong u({\mathfrak b})_{\text{ad}}\cong \overline{S}({\mathfrak b}).$$ 
A basis for $\overline{S}^{n}({\mathfrak b})$ is given by $\langle h^{i}e^{j}:\ \ i+j=n,\ 0\leq i,j \leq p-1 \rangle$.  The monomial $h^{i}e^{j}$ has weight $-2j$. 
One can check that  $\overline{S}^{n}({\mathfrak b})$ is a cyclic $u({\mathfrak b})$-module that is 
generated by $h^{n}$ for $n=0,1,\dots,p-1$ and $e^{n-p+1}h^{p-1}$ for $n=p,p+1,\dots,2p-2$. The highest weight of  $\overline{S}^{n}({\mathfrak b})$ is $0$ for $n=0,1,\dots,p-1$ and 
$-2(n-p+1)$ for $n=p,p+1,\dots,2p-2$. 

These facts imply that one can write, 
\begin{equation}\label{E:sl2identify1} 
\overline{S}^{n}({\mathfrak b})\cong L(n)\otimes -n\ \ \ \ \ \ \ \text{$n=0,1,2,\dots,p-1$} 
\end{equation} 
as a $\text{Dist}(B_{1})$-module where $L(n)$ is a $\text{Dist}(G_{1})$ simple module of highest weight $n$. 
Now one has an isomorphism of  $\text{Dist}(B_{1})$-modules: 
\begin{equation}
\overline{S}^{n}({\mathfrak b})\cong \overline{S}^{2p-2-n}({\mathfrak b})\otimes -2(n-p+1)
\end{equation} 
for $n=p,p+1,\dots,2p-2$. 
By using (\ref{E:sl2identify1}), this translates to 
\begin{equation}\label{E:sl2identify2} 
\overline{S}^{n}({\mathfrak b})\cong L(2p-2-n)\otimes -n\ \ \ \ \ \ \ \ \text{$n=p,p+1,2p-2$} 
\end{equation} 

The next step is to focus on the calculation of $\Hh^{j}({\mathfrak u},\DBone)^{T_{1}}$. First note that $\dim {\mathfrak u}=1$, so $\Hh^{j}({\mathfrak u},\DBone)^{T_{1}}=0$ for $j\geq 2$. 
For $j=0$, one can see that the lowest weight of $\overline{S}^{n}(b)$ is $-2n$ for $n=0,1,\dots,p-1$ and $-2(p-1)$ for $n=p,p+1,\dots 2p-2$. This implies that 
\begin{equation} 
\Hh^{0}({\mathfrak u},\DBone)^{T_{1}}\cong \text{Hom}_{\mathfrak u}(k,\overline{S}({\mathfrak b}))^{T_{1}}\cong k.
\end{equation} 
Now Kostant's theorem (cf. \cite[Theorem 4.1.1]{UGA09}) for ${\mathfrak g}=\mathfrak{sl}_{2}$ says that as a $T$-module:  
\begin{equation}
\Hh^{1}({\mathfrak u},L(\lambda))\cong -s_{\alpha}\cdot \lambda=\lambda+2
\end{equation} 
for $\lambda=0,1,\dots,p-1$ where $s_{\alpha}\in W$.  
One can now apply this to (\ref{E:sl2identify1}) and (\ref{E:sl2identify2}) to deduce that as $T$-module 
\begin{equation} 
\Hh^{1}({\mathfrak u},\DBone)^{T_{1}}\cong \Hh^{1}({\mathfrak u},\overline{S}({\mathfrak b}^{*}))^{T_{1}}\cong 
\Hh^{1}({\mathfrak u},\overline{S}({\mathfrak b}^{*}))^{T_{1}}\cong k
\end{equation} 
One can now use this information to calculate $\HH^{\bullet}(B_{1})$. 

\begin{theorem} Let $p\geq 3$ and let  $B$ be a Borel subgroup of $SL_{2}$. Then as a $k$-algebra and as $T/T_1$-module, 
$$\HH^{\bullet}(B_{1})\cong S^{\bullet}({\mathfrak u}^{*})^{(1)}\otimes \Lambda^{\bullet}(x).$$ 
where the generators consisting of elements in ${\mathfrak u}^{*}$ are of degree $2$, and $x$ has degree $1$ and has weight $0$. 
\end{theorem} 

\begin{proof} One has a spectral sequence which is a slightly modified version of the one in Theorem~\ref{T:spectral3}(b), 
$$E_{2}^{2i,j}=S^{i}({\mathfrak u}^{*})^{(1)}\otimes \Hh^{j}({\mathfrak u},\DBone)^{T_{1}}\Rightarrow \Hh^{i+j}(B_{1},\DBone).$$ 
This spectral sequence lies on two lines for $j=0,1$. 

The first observation is that this spectral sequence collapses. One can see this by using the multiplicative structure of the spectral sequence. Since the spectral sequence is in the first quadrant, the terms $E_{2}^{i,0}$ must be all sent to zero under $d_{2}$. On the other hand, the $E_{2}$ terms is also generated by $E_{2}^{0,1}\cong 
\opH^{1}({\mathfrak u},\overline{S}(({\mathfrak b}))^{T_{1}}$, and $d_{2}:E_{2}^{0,1}\rightarrow E_{2}^{2,0}$. Now one uses the fact the differentials in the spectral sequence 
are $T$-module homomorphism, and weights of $E_{2}^{0,1}$ are $0$, and the weights of $E_{2}^{2,0}$ are $2p$. One can now conclude that $d_{2}=0$ and the spectral sequence collapses. 
Therefore, one has an isomorphism of $B$-algebras for the associated graded algebra obtained from filtration used in constructing the spectral sequence, 
\begin{equation} 
\text{gr } \HH^{\bullet}(B_{1})\cong S^{\bullet}({\mathfrak u}^{*})^{(1)}\otimes \Hh^{\bullet}({\mathfrak u},\overline{S}({\mathfrak b}))^{T_{1}}\cong 
S^{\bullet}({\mathfrak u}^{*})^{(1)}\otimes \Lambda^{\bullet}(x)
\end{equation} 
where $x$ is in degree $1$. 

In order to ``ungrade" the spectral sequence, one can use the ideas from \cite[proof of Theorem 3.1.1]{DNN12}. One needs to show that one can find a subalgebra in $ \HH^{\bullet}(B_{1})$ represented by elements of the 
form $1\otimes x$ in $1\otimes\Hh^{\bullet}({\mathfrak u},\overline{S}({\mathfrak b}))^{T_{1}}$. But, as mentioned beforehand the elements in $1\otimes \Hh^{\bullet}({\mathfrak u},\overline{S}({\mathfrak b}))^{T_{1}}$ 
all have weight zero and the product of any two of these elements cannot be represented in $S^{n}({\mathfrak u}^{*})^{(1)}\otimes  \Hh^{\bullet}({\mathfrak u},\overline{S}({\mathfrak b}))^{T_{1}}$ for 
$n>0$ due to weight considerations. The statement of the theorem now follows and the proof is complete. 
\end{proof} 

\subsection{Case when $p=2$} Note that in this case that $U_{1}$ acts trivially on $\overline{S}^{\bullet}({\mathfrak b})$. This shows that as rings 
$$\Hh^{\bullet}(U_{1},\overline{S}^{\bullet}({\mathfrak b}))\cong \Hh^{\bullet}(U_{1},k)\otimes \overline{S}^{\bullet}({\mathfrak b})\cong 
\Hh^{\bullet}(U_{1},k)\otimes \text{Dist}(B_{1})_{\text{ad}}.$$ 
In characteristic $2$, all the $T$-modules involved above have weights in the root lattice, so everything is $T_{1}$-invariant. 
One has $\opH^{\bullet}(B_{1},k)\cong S^{\bullet}({\mathfrak u}^{*})^{(1)}$, therefore we can conclude the following result. 

\begin{theorem} Let $p=2$ and let  $B$ be a Borel subgroup of $SL_{2}$. Then as a $k$-algebra and as $T/T_{1}$-module, 
$$\HH^{\bullet}(B_{1})\cong S^{\bullet}({\mathfrak u}^{*})^{(1)}\otimes \operatorname{Dist}(B_{1})_{\operatorname{ad}}$$ 
where the generators consisting of elements in ${\mathfrak u}^{*}$ are of degree $1$. 
\end{theorem}

\subsection{Computation of $\HH^{\bullet}(U_{r})$} Once again assume that $U$ is the unipotent radical of $B$ in $G=SL_{2}$. The group $U$ is abelian and the adjoint action of 
$U_{r}$ on $\text{Dist}(U_{r})_{\text{ad}}$ is trivial. Therefore, one has the following isomorphism as a $T$-algebra. 
\begin{equation} 
\HH^{\bullet}(U_{r})\cong \Hh^{\bullet}(U_{r},k)\otimes \text{Dist}(U_{r})_{\text{ad}}.
\end{equation}
Since $U_{r}$ is abelian, the structure of  $\Hh^{\bullet}(U_{r},k)$ can be deduced from \cite[I 9.14 Proposition]{rags}. In particular for the $r=1$ case, one has the following.  
\begin{prop} Let $G=SL_{2}$. One has the isomorphisms as a ring and $T$-module. 
\begin{itemize} 
\item[(a)] For $p=2$, 
$$\HH^{\bullet}(U_{1})\cong S^{\bullet}({\mathfrak u}^{*})^{(1)}\otimes \operatorname{Dist}(U_{1})_{\operatorname{ad}}$$ 
where the generators in the symmetric algebra are in degree one. 
\item[(b)] For $p\geq 3$, 
$$\HH^{\bullet}(U_{1})\cong S^{\bullet}({\mathfrak u}^{*})^{(1)}\otimes \Lambda^{\bullet}({\mathfrak u}^{*})\otimes \operatorname{Dist}(U_{1})_{\operatorname{ad}}$$ 
where the generators in the symmetric algebra are in degree two, and the generators in the exterior algebra are in degree one. 
\end{itemize} 
\end{prop}

\section{$SL_{2}$ computations}\label{S:SL2}

\subsection{Symmetric powers} Let ${\mathfrak g}=\text{Lie }G$ where $G$ is a semisimple simply connected algebraic group. If the characteristic of the field $k$ is good then 
$S^{\bullet}({\mathfrak g}^{*})$ has a good filtration \cite[4.4 Proposition]{AJ84}. Note when $p\geq 3$ for ${\mathfrak g}=\mathfrak{sl}_{2}$, one has ${\mathfrak g}^{*}\cong {\mathfrak g}$, thus 
$S^{\bullet}({\mathfrak g})$ has a good filtration. Furthermore, if ${\mathfrak g}$ is a tilting module then $S^{n}({\mathfrak g})$ is a tilting module for $0\leq n \leq p-1$ \cite[4.1(5)]{AJ84}. 
In the following theorem, we analyze these decompositions when $G=SL_{2}$. 

\begin{prop} \label{P:symmetricpowers} Let $G=SL_{2}$, ${\mathfrak g}=\mathfrak{sl}_{2}$ and $p\geq 3$.
\begin{itemize}
\item[(a)] In the Grothendieck group, one has 
$$[S^{n}({\mathfrak g})]=[\nabla(2n)]+[\nabla(2n-4)]+[\nabla(2n-8)]+\dots+[\nabla(2n-4q)]$$ 
where $q$ is the largest integer such that $2n-4q\geq 0$
\item[(b)] Let $p\geq 3$. 
\begin{itemize} 
\item[(i)] For $0\leq n \leq \frac{p-1}{2}$, there exists an isomorphism of $G$-modules
$$S^{n}({\mathfrak g})\cong T(2n)\cong L(2n).$$ 
\item[(ii)] For $\frac{p-1}{2}<n\leq p-1$, there exists an isomorphism of $G$-modules 
$$S^{n}({\mathfrak g})\cong T(2n)\oplus T(2n-4) \oplus \dots \oplus T(2n-4r)\oplus T(2n-4s)\oplus  T(2n-4(s+1)),\dots \oplus T(2n-4t)
$$ 
where $r$ is the largest integer with $2n-4r\geq p$, $s$ is the largest integer such that $2n-4s < 2p-2-2n$ and $t$ is the largest integer such that $2n-4t\geq 0$. 
\end{itemize} 
\end{itemize} 
\end{prop} 

\begin{proof} (a) Fix $n\geq 0$, then $S^{n}({\mathfrak g})$ has a good filtration since $p\geq 3$. One can show by induction that the character of $S^{n}({\mathfrak g})$ is the 
sum of the characters of $S^{n-2}({\mathfrak g})$ and $\nabla(2n)$. This can be seen as follows. One can multiply the elements of $S^{n-2}({\mathfrak g})$ by $h^{2}$ and this does not 
change the character, and yields a $T$-submodule of $S^{n}({\mathfrak g})$ with monomials of the form $e^{n_{1}}h^{m}f^{n_{2}}$ where 
$n_{1}=n_{2}$ and $n=n_{1}+n_{2}+m$ with $m\geq 2$. One can then prove directly that the other monomials where $m=0,1$ form a $T$-submodule with character 
$\nabla(2n)$. This proves the first statement. 

(b) For $p\geq 3$, ${\mathfrak g}$ is a tilting module for $G$. Therefore, $S^{n}({\mathfrak g})$ is a tilting module for $G$ when $0\leq n \leq p-1$. For (i), when 
 $0\leq n \leq \frac{p-1}{2}$, the highest weight is restricted (and in the bottom alcove), so $S^{n}({\mathfrak g})\cong T(2n)=L(2n)$. 
 
 The tilting module $T(\lambda)$ when $p\leq \lambda \leq 2p-2$ has a good filtration with two factors $\nabla(\lambda)$ and $\nabla(\mu)$ where $\lambda+\mu=2p-2$. 
This fact uses the linkage principle for $G$.  In case (ii), one 
 can use (a), to see that $S^{n}({\mathfrak g})$ will have tilting factors $T(2n), T(2n-4),\dots, T(2n-4r)$ where $r$ is the largest integer with $2n-4r\geq p$. Moreover, there will be possibly other simple tilting module 
 factors $T(2n-4s), T(2n-4(s+1)),\dots T(2n-4t)$ where $s$ is the largest integer such that $2n-4s < 2p-2-2n$ and $t$ is the largest integer such that $2n-4t\geq 0$. 
 
\end{proof} 

\subsection{Structure of Truncated Symmetric Powers} Let ${\mathfrak a}$ be an arbitrary restricted Lie algebra over $k$ with $\overline{S}({\mathfrak a}):=\overline{S}^{\bullet}({\mathfrak a})$ be the truncated symmetric algebra which is obtained by taking the quotient of the ordinary symmetric algebra $S^{\bullet}({\mathfrak a})$ with the ideal generated by $p$th powers. The Lie algebra ${\mathfrak a}$ acts on $\overline{S}({\mathfrak a})$ via the adjoint representation. The highest degree term $\overline{S}^{N}({\mathfrak a})$ is one-dimensional and occurs when $N=(p-1)\dim {\mathfrak a}$. Let $\chi$ denote this one-dimensional 
$u({\mathfrak a})$-module. 

For each $i=0,1,\dots, N$, there exists a $u({\mathfrak a})$ homomorphism  
\begin{equation}
(-,-): \overline{S}^{i}({\mathfrak a})\otimes  \overline{S}^{N-i}({\mathfrak a})\rightarrow \chi
\end{equation} 
given by $(f,g)=f\cdot g$. This pairing is non-degenerate in both variables so there exists an isomorphism of $u({\mathfrak a})$-module 
\begin{equation} 
\overline{S}^{i}({\mathfrak a})\cong \overline{S}^{N-i}({\mathfrak a})^{*}\otimes \chi. 
\end{equation} 
One can upgrade this to an $A$-module isomorphism when ${\mathfrak a}=\text{Lie }A$. 

For the case when $G=SL_{2}$, $\chi$ is the trivial module. Therefore, for 
$i=0,1,2,\dots,N$, one has an isomorphism as $G$-module: 
\begin{equation} \label{E:duality}
\overline{S}^{i}({\mathfrak g})\cong \overline{S}^{N-i}({\mathfrak g})^{*}. 
\end{equation} 

\subsection{The $G$-module structure of $\HH^{\bullet}(G_{1})$} Assume that $p\geq 3$. For $G=SL_2$, the weights of $\overline{S}({\mathfrak g})$ are less than $2(p-1)=2p-2$. So we can consider the 
truncated category ${\mathcal C}_{\pi}$ which is the full subcategory of $G$-modules whose composition factor have highest weights in $\pi=\{\lambda\in X^{+}:\ 0\leq \lambda \leq 2(p-1)\}$. 

Let $\overline{S}({\mathfrak g})_{0}$ be the projection of $\overline{S}({\mathfrak g})$ onto the principal block for $G$. This will be the same as the projection onto the $G_{1}$-block because 
$p$ is odd and all the weights of $\overline{S}({\mathfrak g})$ are even. The $G$-module, $S^{n}({\mathfrak g})$ has a good filtration and the induced modules in the principal $G$-block are $\nabla(0)$, $\nabla (2p-2)$ for $0\leq n\leq 3(p-1)$. Now one can use Proposition~\ref{P:symmetricpowers} and (\ref{E:duality}) to obtain the following calculation. 

\begin{prop} \label{P:S-factors} Let $p\geq 3$. The $G$-module structure on  $\overline{S}^{\bullet}({\mathfrak g})_{0}$ is given by 
\begin{itemize} 
\item[(a)] $\overline{S}^{n}({\mathfrak g})_{0}\cong k$ for $n$ even, $0\leq n< p-1$, $2(p-1)<n \leq 3(p-1)$. 
\item[(b)] $\overline{S}^{n}({\mathfrak g})_{0}\cong 0$ for $n$ odd, $0\leq n< p-1$, $2(p-1)<n \leq 3(p-1)$. 
\item[(c)] $\overline{S}^{n}({\mathfrak g})_{0}\cong T(2p-2)$ for $n$ even, $p-1\leq n \leq 2(p-1)$. 
\item[(d)] $\overline{S}^{n}({\mathfrak g})_{0}\cong L(2p-2)$ for $n$ odd, $p-1\leq n \leq 2(p-1)$. 
\end{itemize} 
\end{prop} 

We now review the cohomological calculations given in \cite{AJ84}. First note that 
\begin{equation}
T(2p-2)|_{G_{1}}\cong Q_{1}(0)
\end{equation} 
where $Q_{1}(0)$ is the $G_{1}$- projective cover of the trivial module. 
Hence, 
\begin{equation} \label{E:coho1}
\opH^{n}(G_{1},T(2p-2))^{(-1)}= 
\begin{cases}
k & \text{for $n=0$}\\
0 &\text{ else.}
\end{cases}
\end{equation} 
Moreover, from \cite[Corollary 3.7, Section 3.9]{AJ84}, one has the computation of the cohomology ring for $G_{1}$ for $p\geq 3$: 
\begin{equation} \label{E:coho2}
\opH^{n}(G_{1},k)^{(-1)}= 
\begin{cases}
k[{\mathcal N}]_{\frac{n}{2}} & \text{for $n$ even}\\
0 & \text{for $n$ odd}.
\end{cases}
\end{equation} 
One has by Steinberg's tensor product formula: $L(2p-2)\cong L(p-2)\otimes L(1)^{(1)}$. Moreover, $\nabla(p-2)\cong L(2p-2)$, so by \cite[Corollary 3.7]{AJ84}, 
\begin{equation} \label{E:coho3} 
\opH^{n}(G_{1},L(2p-2))^{(-1)}= 
\begin{cases}
 \text{ind}_{B}^{G} [S^{\frac{n-1}{2}}({\mathfrak u}^{*})\otimes \omega] \otimes L(1) &  \text{for $n$ odd}\\
0 & \text{for $n$ even}
\end{cases}
\end{equation} 
We are now ready to provide the $G$-module structure for the Hochschild cohomology groups $\HH^{\bullet}(G_{1})$ for $G=SL_2$. 

\begin{theorem} Let $G=SL_{2}$ with $p\geq 3$. As a $G$-module, the Hochschild cohomology is isomorphic to the following modules.  
\begin{itemize}
\item[(a)] $\HH^{0}(G_{1})^{(-1)}\cong k^{\oplus \frac{p-1}{2}}$; 
\item[(b)] $\HH^{2\bullet}(G_{1})^{(-1)} \cong k[{\mathcal N}]_{\bullet}^{\oplus (p-1)}$ for $\bullet > 0$; 
\item[(c)] $\HH^{2\bullet+1}(G_{1})^{(-1)} \cong [\operatorname{ind}_{B}^{G} [S^{\bullet}({\mathfrak u}^{*})\otimes \omega]\otimes L(1)]^{\oplus \frac{p-1}{2}}$ for $\bullet \geq  0$. 
\end{itemize} 
\end{theorem} 

\begin{proof} First recall that 
\begin{equation} 
\HH^{\bullet}(G_{1})=\text{Ext}^{\bullet}_{G_{1}}(k,\overline{S}({\mathfrak g}))\cong \opH^{\bullet}(G_{1},\overline{S}({\mathfrak g}))\cong 
\opH^{\bullet}(G_{1},\overline{S}({\mathfrak g})_{0}).
\end{equation} 
From Proposition~\ref{P:S-factors}, the possible non-zero $G$-summands of $\overline{S}({\mathfrak g})_{0}$ are $k$, $T(2p-2)$ and $L(2p-2)$, and the $G_{1}$-cohomology for these summands is given in 
(\ref{E:coho1}), (\ref{E:coho2}), and (\ref{E:coho3}). The statement of the theorem now follows by keeping track of the multiplicities of the $G$-factors which can be found in 
Proposition~\ref{P:S-factors}. 
\end{proof} 

The case when $p=2$ requires a slightly different argument. 

\begin{theorem}\label{T:G-module} Let $G=SL_{2}$ with $p=2$. As a $G$-module, the Hochschild cohomology is isomorphic to the following modules.  
\begin{itemize}
\item[(a)] $\HH^{0}(G_{1})^{(-1)}\cong k^{\oplus 3}\oplus L(1)$; 
\item[(b)] $\HH^{\bullet}(G_{1})^{(-1)} \cong k[{\mathcal N}]_{\bullet}^{\oplus 2}\oplus \operatorname{ind}_{B}^{G} [S^{\bullet}({\mathfrak u}^{*})^{(-1)}\otimes \omega] \oplus 
 \operatorname{ind}_{B}^{G} [S^{\bullet}({\mathfrak u}^{*})^{(-1)}\otimes -\omega]$ for $\bullet > 0$.
\end{itemize}
\end{theorem} 

\begin{proof} The structure of $\bar{S}({\mathfrak g})$ for $p=2$ is given in the first table in the Appendix. The possible $G$-factors are $k$, $\nabla(2)$ and $\Delta(2)$ which can be obtained by using 
Proposition~\ref{P:symmetricpowers} and  (\ref{E:duality}). 

We now look at the $B_{1}$-cohomology groups, $\operatorname{H}^{\bullet}(B_{1}, \bar{S}({\mathfrak g}))^{(-1)}$ as a $B$-module. First, one has as a $B$-module, 
\begin{equation} 
\operatorname{H}^{\bullet}(B_{1}, \bar{S}^{n}({\mathfrak g}))^{(-1)}\cong S^{\bullet}({\mathfrak u}^{*})^{(-1)}
\end{equation} 
for $n=0,3$ by \cite[2.4(3)]{AJ84}. 
Next, observe that $\nabla(2)$ as a $B_{1}T$-module is the direct sum of the one-dimensional module of weight $2$ and the projective indecomposable $B_{1}T$-module for the weight $-2$. 
Similarly, $\Delta(2)$ as a $B_{1}T$-module is the direct sum of the one-dimensional module of weight $-2$ and the projective indecomposable $B_{1}T$-module for the weight $0$. 
From this information one can deduce that for $\bullet>0$, one has the following isomorphisms as $T$-modules: 
\begin{equation} 
\operatorname{H}^{\bullet}(B_{1}, \nabla(2))^{(-1)}\cong S^{\bullet}({\mathfrak u}^{*})^{(-1)}\otimes \omega
\end{equation} 
\begin{equation} 
\operatorname{H}^{\bullet}(B_{1}, \Delta(2))^{(-1)}\cong S^{\bullet}({\mathfrak u}^{*})^{(-1)}\otimes -\omega.
\end{equation} 

Finally, we need to handle the case for $\operatorname{H}^{0}(B_{1},\nabla(2))$ and $\operatorname{H}^{0}(B_{1},\Delta(2))$. As a $B_{1}T$-modules, one has 
\begin{equation} 
\operatorname{H}^{0}(B_{1},\nabla(2))^{(-1)}\cong \omega \oplus -\omega. 
\end{equation} 
\begin{equation}
\operatorname{H}^{0}(B_{1},\Delta(2))^{(-1)}\cong -\omega \oplus k. 
\end{equation} 
In order to obtain the $B$-module structure, one can use the isomorphism: 
\begin{equation} 
\text{Hom}_{B}(\mu,\operatorname{H}^{0}(B_{1},N)^{(-1)})\cong \text{Hom}_{B}(2\mu,N)\cong \text{Hom}_{G}(\Delta(-2\mu),N). 
\end{equation} 
Plugging in $N=\nabla(2)$ and $\Delta(2)$, one obtains the following isomorphisms as $B$-modules: 
\begin{equation} 
\operatorname{H}^{0}(B_{1},\nabla(2))^{(-1)}\cong L(1).  
\end{equation} 
\begin{equation} 
\operatorname{H}^{0}(B_{1},\Delta(2))^{(-1)}\cong -\omega \oplus k. 
\end{equation} 
Since the weights of $\operatorname{H}^{\bullet}(B_{1}, \bar{S}^{n}({\mathfrak g}))^{(-1)}\cong S^{\bullet}({\mathfrak u}^{*})^{(-1)}$ are greater than or equal to $-1$, one has by using \cite[II 4.5 Proposition, 5.4 Proposition]{rags}
\begin{equation} 
\opH^{\bullet}(G_{1},\bar{S}({\mathfrak g}))\cong \text{ind}_{B}^{G} \opH^{\bullet}(B_{1}, \bar{S}({\mathfrak g})).
\end{equation} 
The statement of the theorem now follows by using this isomorphism with the calculations of the $B_{1}$-cohomology. 
\end{proof} 

\subsection{The $G$-algebra structure of $\HH^{\bullet}(G_{1})$} We first describe the ring structure of the ${\mathfrak u}$-cohomology for 
$\bar{S}({\mathfrak g})_{0}$ under the $T_{1}$-invariants. 

\begin{prop} Let $p\geq 3$ and $G=SL_{2}$. Then 
\begin{itemize} 
\item[(a)] $\operatorname{H}^{\bullet}({\mathfrak u},\bar{S}^{\bullet}({\mathfrak g})_{0})^{T_{1}}$ has a $k$-basis with elements 
\begin{itemize} 
\item[(i)] $\{1, x, x^{2},\dots,x^{s}\}$ where $x$ has bidegree (0,2), and $s=\frac{3(p-1)}{2}$, 
\item[(ii)] $\{y, yx, yx^{2},\dots,yx^{t}\}$ where $y$ has bidegree $(1,p-1)$, and $t=\frac{p-1}{2}$, 
\item[(iii)] $\{z, zx,\dots, zx^{u}\}$ where $z$ has bidegree $(1,p)$, and $u=\frac{p-3}{2}$.  
\item[(iv)] $\{z^{\prime}, z^{\prime}x,\dots, z^{\prime}x^{u}\}$ where $z^{\prime}$ has bidegree $(1,p)$, and $u=\frac{p-3}{2}$.  
\end{itemize}
\item[(b)] The element $x$ is represented by $1\otimes [ef-h]$, $y$ is represented by $e^{*}\otimes f^{p-1}$, and $z$ is represented by 
$e^{*}\otimes f^{p-1}h$. 
\item[(c)] The multiplication is given by the following. 
\begin{itemize} 
\item[(i)] The elements in (i)-(iv)  in (a) commute. 
\item[(ii)] The product of the elements in (ii)-(iv)  in (a) are zero. 
\end{itemize} 
\item[(d)] The weights of the elements in (i), (ii), and (iv) are $0$, and in (iii) are $2p$. 
\end{itemize} 
\end{prop} 

\begin{proof}  In order to verify (b), (c), and (d), we observe that $\opH^{\bullet}({\mathfrak u},\bar{S}^{\bullet}({\mathfrak g})_{0})^{T_{1}}$ is a subquotient of 
$$\opH^{\bullet}({\mathfrak u},k)\otimes \bar{S}^{\bullet}({\mathfrak g})_{0})^{T_{1}}\cong 
[\Lambda^{\bullet}({\mathfrak u}^{*}) \otimes \bar{S}^{\bullet}({\mathfrak g})_{0}]^{T_{1}}.$$ 
The bigraded ring structure of $\Lambda^{\bullet}({\mathfrak u}^{*})\otimes S^{\bullet}({\mathfrak g})_{0}$ arises directly from the ring structures on the graded rings, 
$\Lambda^{\bullet}({\mathfrak u}^{*})$ and $S^{\bullet}({\mathfrak g})$. One has $\dim \Lambda^{n}({\mathfrak u}^{*})=0$ for $n\geq 2$, and 
$\dim \Lambda^{n}({\mathfrak u}^{*})=1$ for $n=0,1$. The $T$-module $\Lambda^{0}({\mathfrak u}^{*})$ is spanned by $\{1\}$, and 
$\Lambda^{1}({\mathfrak u}^{*})$ is spanned by $\{e^{*}\}$.  

Next we describe a basis for  $\operatorname{H}^{\bullet}({\mathfrak u},\bar{S}^{\bullet}({\mathfrak g})_{0})^{T_{1}}$ to justify part (a).  According to Proposition~\ref{P:S-factors}, 
the $G$-summands of $\bar{S}^{\bullet}({\mathfrak g})_{0}$ are $k$, $T(2p-2)$, and $L(2p-2)$. For (i), the element $x$ is represented by $1\otimes [ef-h^{2}]$, and the powers of $x$ are a basis for 
terms involving $\opH^{0}({\mathfrak u},k)^{T_{1}}$ and $\opH^{0}({\mathfrak u},T(2p-2))^{T_{1}}$. Note that $\opH^{1}({\mathfrak u},k)^{T_{1}}=0$. 

For (ii), $y$ is represented by $e^{*}\otimes f^{p-1}$. The elements in (ii) form a basis for the terms involving $\opH^{1}({\mathfrak u},T(2p-2))^{T_{1}}$. Finally, note that 
$$\opH^{0}({\mathfrak u},L(2p-2))^{T_{1}}\cong \opH^{0}({\mathfrak u},L(p-2))^{T_{1}}\otimes L(1)^{(1)}=0.$$ 
A basis for  
$\opH^{1}({\mathfrak u},L(2p-2))^{T_{1}}\cong \opH^{1}({\mathfrak u},L(p-2))^{T_{1}}\otimes L(1)^{(1)}$ is given by the elements listed in (iii) and (iv). 
\end{proof} 

We remark that the element $z^{\prime}$ for $p=3$ can be represented by $e^{*}\otimes [eh^{2}-e^{2}f]$. For $p\geq 5$ this expression becomes more complicated. 
Next, we need to understand the ring structure of $\opH^{\bullet}(B_{1},\bar{S}({\mathfrak g})_{0})$. 

\begin{theorem}\label{T:B1-coho} Let $G=SL_{2}$. As a $B$-algebra, one has the following isomorphisms. 
\begin{itemize} 
\item[(a)] For $p\geq 3$, 
$$\opH^{\bullet}(B_{1},\bar{S}({\mathfrak g})_{0})\cong [S^{\bullet}({\mathfrak u}^{*})^{(1)}\otimes  \operatorname{H}^{\bullet}({\mathfrak u},\bar{S}({\mathfrak g})_{0})^{T_{1}}]/I_{p}$$
where elements in ${\mathfrak u}^{*}$ have degree $2$ and 
$$I_{p}=\langle \bigoplus_{i=0}^{(p-1)/2}S^{\bullet}({\mathfrak u}^{*})^{(1)}\otimes f^{p-1}x^{i},\  \bigoplus_{j=(p-1)/2}^{p-1}S^{\bullet}({\mathfrak u}^{*})^{(1)}_{+}\otimes x^{j} \rangle. $$
\item[(b)]  For $p=2$, 
$$\opH^{\bullet}(B_{1},\bar{S}({\mathfrak g})_{0})\cong [S^{\bullet}({\mathfrak u}^{*})\otimes \bar{S}({\mathfrak g})]/I_{2}$$ 
where the elements in ${\mathfrak u}^{*}$ have degree $1$, and 
$$
I_{2}=\langle S^{\bullet}({\mathfrak u}^{*})^{(1)}\otimes ef,\  S^{\bullet}({\mathfrak u}^{*})^{(1)}\otimes f,\  S^{\bullet}({\mathfrak u}^{*})^{(1)}_{+}\otimes h,\  
S^{\bullet}({\mathfrak u}^{*})^{(1)}_{+}\otimes fh \rangle. 
$$ 
\end{itemize} 
\end{theorem} 
\begin{proof}  Let $p\geq 3$. For any $B_{1}$-module $N$, one has a spectral sequence (cf. Theorem~\ref{T:spectral3}): 
\begin{equation}\label{eq:ss-B1coho} 
E_{2}^{2i,j}=S^{i}({\mathfrak u}^{*})^{(1)}\otimes \Hh^{j}({\mathfrak u},N)^{T_{1}}\Rightarrow \opH^{i+j}(B_{1},N).
\end{equation}  
We will apply this spectral sequence to the $G$-summands of $\bar{S}({\mathfrak g})_{0}$ which are 
$k$, $L(2p-2)$, and $T(2p-2)$. 

For $N=k$, $\Hh^{\bullet}({\mathfrak u},k)^{T_{1}}\cong k$. The spectral sequence collapses (\ref{eq:ss-B1coho}) at $E_{2}$, and 
\begin{equation} 
\opH^{\bullet}(B_{1},k)^{(-1)}\cong S^{\bullet/2}({\mathfrak u}^{*}). 
\end{equation} 
For $N=L(2p-2)$, observe that $L(2p-2)=L(p-2)\otimes L(1)^{(1)}\cong \nabla(p-2)\otimes L(1)^{(1)}$, thus by \cite[Corollary 3.7]{AJ84}
\begin{equation} 
\opH^{\bullet}(B_{1},L(2p-2))^{(-1)}\cong [S^{(\bullet-1)/2}({\mathfrak u}^{*})\otimes \omega]\otimes L(1) . 
\end{equation} 
In this case the spectral sequence (\ref{eq:ss-B1coho}) when $N=L(2p-2)$ collapses at $E_{2}$. 
For $N=T(2p-2)$, one has $T(2p-2)|_{G_{1}}\cong Q_{1}(0)$ which is projective as a $G_{1}$-module, thus also as a $B_{1}$-module. 
Observe that $\opH^{j}(B_{1},N)=0$ for $j>0$ and $\opH^{0}(B_{1},N)\cong k$. Therefore, the spectral sequence (\ref{eq:ss-B1coho}) 
cannot collapse at $E_{2}$. Indeed, by considering the differentials, the spectral sequence must stop at $E_{3}$, for these terms. 

Now one can put this all together for $N=\bar{S}^{\bullet}({\mathfrak g})_{0}$. The spectral sequence will stop at $E_{3}$, and yield an isomorphism: 
$$\text{gr }\opH^{\bullet}(B_{1},\bar{S}({\mathfrak g})_{0})\cong [S^{\bullet}({\mathfrak u}^{*})^{(1)}\otimes  \operatorname{H}^{\bullet}({\mathfrak u},\bar{S}({\mathfrak g})_{0})^{T_{1}}]/I_{p}.$$
The terms in $I_{p}$ arise as the elements in the first and second rows in the spectral sequence which are eliminated by the differentials $\delta_{2}$ for the summands $T(2p-2)$
in $N$. 

We need to degrade the spectral sequence. First note as a ring 
\begin{equation}
\opH^{\bullet}(U_{1},k)\cong  S^{\bullet} ({\mathfrak u}^{*})^{(1)}\otimes \opH^{\bullet}({\mathfrak u},k).
\end{equation} 
 For $G=SL_{2}$ and 
$p\geq 3$, this is not hard to verify. However, for general reductive groups $G$, this is known to hold when $p\geq 2h-1$ and requires an argument to degrade a spectral sequence. 
Using this fact, one has as a $B/B_{1}$-algebra: 
\begin{eqnarray*} 
\opH^{\bullet}(B_{1},\bar{S}({\mathfrak g})_{0})&\cong &[\opH^{\bullet}(U_{1},k)\otimes \bar{S}({\mathfrak g})_{0})]^{T_{1}}\\
&\cong& S^{\bullet}({\mathfrak u}^{*})^{(1)}\otimes [\opH^{\bullet}({\mathfrak u},k)\otimes \bar{S}({\mathfrak g})_{0}]^{T_{1}}
\end{eqnarray*} 
Now the ring structure of $\operatorname{H}^{\bullet}({\mathfrak u},\bar{S}({\mathfrak g})_{0})^{T_{1}}$ arises from that of 
 $[\opH^{\bullet}({\mathfrak u},k)\otimes \bar{S}({\mathfrak g})_{0}]^{T_{1}}$. Therefore, one can ungrade the spectral sequence and we can deduce that as a $B/B_{1}$-algebra: 
 $$\opH^{\bullet}(B_{1},\bar{S}({\mathfrak g})_{0})\cong [S^{\bullet}({\mathfrak u}^{*})^{(1)}\otimes  \operatorname{H}^{\bullet}({\mathfrak u},\bar{S}({\mathfrak g})_{0})^{T_{1}}]/I_{p}.$$
 
For the $p=2$ case, observe that $\bar{S}^{\bullet}({\mathfrak g})\cong \bar{S}^{\bullet}({\mathfrak g})_{0}$ and 
$$\opH^{\bullet}(B_{1},\bar{S}({\mathfrak g}))\cong \opH^{\bullet}(U_{1},\bar{S}({\mathfrak g}))^{T_{1}}\cong 
\opH^{\bullet}(U_{1},,k)^{T_{1}}\otimes \bar{S}({\mathfrak g})\cong  \opH^{\bullet}(B_{1},,k)\otimes \bar{S}({\mathfrak g}) .$$ 
Note that all modules above are $T_{1}$-invariant. 
This shows that the ring $\opH^{\bullet}(B_{1},\bar{S}({\mathfrak g}))$ is a subquotient of the ring $S^{\bullet}({\mathfrak u}^{*})^{(1)}\otimes \bar{S}({\mathfrak g})$. 
Now one can use the $B$-module structure of $\opH^{\bullet}(B_{1},\bar{S}({\mathfrak g}))^{(-1)}$ given in the proof of Theorem~\ref{T:G-module}, to see that 
$$\opH^{\bullet}(B_{1},\bar{S}({\mathfrak g})_{0})\cong [S^{\bullet}({\mathfrak u}^{*})\otimes \bar{S}({\mathfrak g})]/I_{2}.$$ 
\end{proof} 

We can now provide a description of the ring structure of $\HH^{\bullet}(G_{1})$ via the induction functor and the explicit description of 
$\operatorname{H}^{\bullet}(B_{1},\bar{S}({\mathfrak g}))$ provided in Theorem~\ref{T:B1-coho}. 

\begin{theorem} Let $p\geq 2$ and $G=SL_{2}$. Then there exists an isomorphism of rings 
$$ \HH^{\bullet}(G_{1})^{(-1)}\cong \operatorname{ind}_{B}^{G} \operatorname{H}^{\bullet}(B_{1},\bar{S}({\mathfrak g})_{0})^{(-1)}.$$ 
\end{theorem} 

\begin{proof} According to Theorem~\ref{T:spectral1}, there exists a spectral sequence: 
$$E_{2}^{i,j}=R^{i}\operatorname{ind}_{B}^{G} \Hh^{\bullet}(B_{1},\bar{S}({\mathfrak g})_{0})^{(-1)} \Rightarrow \HH^{i+j}(G_{1})^{(-1)}.$$ 
From Theorem~\ref{T:G-module}, it follows that as a $B$-module, $\opH^{\bullet}(B_{1},\bar{S}({\mathfrak g})_{0})^{(-1)}$ is a direct sum of $B$-module factors isomorphic to $k$, and  
$S^{\bullet}({\mathfrak u}^{*})$, $\omega$, and $-\omega$ (when $p=2$). From Kempf's vanishing theorem, and \cite[II Proposition 5.4]{rags}, the higher cohomologies, $R^{j}\text{ind}_{B}^{G}(-)$ vanish on these factors. Hence, the spectral sequence collapses and yields the result. 
\end{proof} 

\section{Appendix}\label{S:appendix} 

The following tables give the $G$-module structure of $\HH^{\bullet}(G_{1})$ for $G=SL_{2}$ for $p=2,3,5,7$. Note that 
$\text{Dist}(G_{1})_{\text{ad}}\cong  \bigoplus_{n=0}^{3(p-1)} \overline{S}^{n}({\mathfrak g})$, and as $G$-module 
$$\HH^{\bullet}(G_{1})\cong \bigoplus_{n=0}^{3(p-1)} \Hh^{\bullet}(G_{1},\overline{S}^n({\mathfrak g})).$$ 
In the tables in the last column, if a copy of the trivial module $k$ occurs then the trivial module is located in degree zero of 
$\Hh^{\bullet}(G_{1},\overline{S}^n({\mathfrak g}))^{(-1)}$. Here $\omega$ is the first fundamental weight and $L(1)$ is the $2$-dimensional natural representation. 
\vskip .5cm 

\subsection{$p=2$ case} 

\begin{center}
   \begin{tabular}{ c|c|c|c } 
$n$ & $\overline{S}^n({\mathfrak g})$ & $\Hh^{\bullet}(G_{1},\overline{S}^n({\mathfrak g}))^{(-1)}$  &  \  \\
\hline
0 & $T(0)$& $k[{\mathcal N}]_{\bullet}$ & \  \\ 
\hline
1 & $\Delta(2)$ & $\operatorname{ind}_{B}^{G} [S^{\bullet}({\mathfrak u}^{*})^{(-1)}\otimes -\omega]$ & \text{for $\bullet >0$} \\
\hline 
1 & $\Delta(2)$ &$k$ &  \text{for $\bullet=0$} \\
\hline
2 & $\nabla(2)$ & $\operatorname{ind}_{B}^{G} [S^{\bullet}({\mathfrak u}^{*})^{(-1)}\otimes \omega]$ &  \text{for $\bullet >0$} \\
\hline 
2 & $\nabla(2)$ &$L(1)$  & \text{for $\bullet=0$} \\
\hline
3 & $T(0)$ &$k[{\mathcal N}]_{\bullet}$ \\
\hline
\end{tabular} 
\end{center}

\subsection{$p=3$ case} 

\begin{center}
   \begin{tabular}{ c|c|c } 
$n$ & $\overline{S}^n({\mathfrak g})$ & $\Hh^{\bullet}(G_{1},\overline{S}^n({\mathfrak g}))^{(-1)}$ \\
\hline
0 & $T(0)$& $k[{\mathcal N}]_{\frac{\bullet}{2}}$\\ 
\hline
1 & $T(2)$ &$0$\\
\hline
2 & $T(4)$& $k$ \\ 
\hline
3 & $L(4)\oplus T(2)$ &$\text{ind}_{B}^{G} [S^{\frac{\bullet-1}{2}}({\mathfrak u}^{*})\otimes \omega] \otimes L(1)$\\
\hline
4 & $T(4) $ &$k$ \\
\hline
5 & $T(2)$&$0$\\
\hline
6 & $T(0)$& $k[{\mathcal N}]_{\frac{\bullet}{2}}$\\
\hline

\end{tabular} 
\end{center}

\subsection{$p=5$ case} 

\begin{center}
   \begin{tabular}{ c|c|c } 
$n$ & $\overline{S}^n({\mathfrak g})$ & $\Hh^{\bullet}(G_{1},\overline{S}^n({\mathfrak g}))^{(-1)}$ \\
\hline
0 & $T(0)$& $k[\mathcal{N}]_{\frac{\bullet}{2}}$\\ 
\hline
1 & $T(2)$ &$0$\\
\hline
2 & $T(4)\oplus T(0)$& $k[\mathcal{N}]_{\frac{\bullet}{2}}$ \\ 
\hline
3 & $T(6)$ &$0$\\
\hline
4 & $T(8)\oplus T(4) $ &$k$ \\
\hline
5 & $L(8)\oplus T(6)$& $\text{ind}_{B}^{G} [S^{\frac{\bullet-1}{2}}({\mathfrak u}^{*})\otimes \omega] \otimes L(1)$\\
\hline
6 & $L(6)\bigoplus T(8)\bigoplus T(4)$&$k$\\
\hline
7 & $L(8)\oplus T(6)$ & $\text{ind}_{B}^{G} [S^{\bullet}({\mathfrak u}^{*})\otimes \omega] \otimes L(1)$\\
\hline
8 & $T(8)\oplus T(4) $ &$k$ \\
\hline
9 & $T(6)$ &$0$\\
\hline
10 & $T(4)\oplus T(0)$& $k[\mathcal{N}]_{\frac{\bullet}{2}}$ \\ 
\hline
11 & $T(2)$ &$0$\\
\hline
12 & $T(0)$& $k[\mathcal{N}]_{\frac{\bullet}{2}}$\\ 
\hline
\end{tabular} 
\end{center}
\newpage 

\subsection{$p=7$ case} 

\begin{center}
   \begin{tabular}{ c|c|c } 
$n$ & $\overline{S}^n({\mathfrak g})$ & $\Hh^{\bullet}(G_{1},\overline{S}^n({\mathfrak g}))^{(-1)}$ \\
\hline
0 & $T(0)$& $k[{\mathcal N}]_{\frac{\bullet}{2}}$\\ 
\hline
1 & $T(2)$ &$0$\\
\hline
2 & $T(4)\oplus T(0)$ & $k[{\mathcal N}]_{\frac{\bullet}{2}}$ \\ 
\hline
3 & $T(6)\oplus T(2)$ &$0$\\
\hline
4 & $T(8)\oplus T(0) $ &. $k[{\mathcal N}]_{\frac{\bullet}{2}}$ \\
\hline
5 & $T(10)\oplus T(6)$&$0$\\
\hline
6 & $T(12)\oplus T(8)$&$k$\\
\hline
7 & $L(12)\oplus T(10)\oplus T(6)$& $\text{ind}_{B}^{G} [S^{\frac{\bullet-1}{2}}({\mathfrak u}^{*})\otimes \omega] \otimes L(1)$\\
\hline
8 & $L(10)\oplus T(12)\oplus T(8)$ &$k$ \\
\hline
9 & $L(8)\oplus L(12)\oplus T(10)\oplus T(6)$ &$\text{ind}_{B}^{G} [S^{\frac{\bullet-1}{2}}({\mathfrak u}^{*})\otimes \omega] \otimes L(1)$\\
\hline
10 & $L(10)\oplus T(12)\oplus T(8)$ &$k$ \\
\hline
11 & $L(12)\oplus T(10)\oplus T(6)$&$\text{ind}_{B}^{G} [S^{\frac{\bullet-1}{2}}({\mathfrak u}^{*})\otimes \omega] \otimes L(1)$\\
\hline
12 & $T(12)\oplus T(8)$&$k$\\
\hline
13 & $T(10)\oplus T(6)$&$0$\\
\hline
14 & $T(8) \oplus T(0) $ &   $k[{\mathcal N}]_{\frac{\bullet}{2}}$         \\
\hline
15 & $L(6)\oplus T(2)$ &$0$\\
\hline
16 & $L(4)\oplus T(0)$& $k[{\mathcal N}]_{\frac{\bullet}{2}}$\  \\ 
\hline
17 & $L(2)$ &$0$\\
\hline
18 & $T(0)$& $k[{\mathcal N}]_{\frac{\bullet}{2}}$\\ 
\hline
\end{tabular} 
\end{center}


\providecommand{\bysame}{\leavevmode\hbox
to3em{\hrulefill}\thinspace}

\end{document}